\documentclass{amsart}

\usepackage{amssymb}
\usepackage[all]{xy}
\usepackage{epsfig} 
\usepackage{tabularx}

\newcolumntype{C}{>{\centering\arraybackslash}X}

\def\comment#1{{\sf{[#1]}}}

\def\Z{{\mathbb Z}}
\def\Q{{\mathbb Q}}
\def\R{{\mathbb R}}
\def\C{{\mathbb C}}
\def\P{{\mathbb P}}
\def\F{{\mathbb F}}

\def\V{{\mathbb V}}
\def\W{{\mathbb W}}

\def\A{{\mathcal A}}
\def\cC{{\mathcal C}}
\def\D{{\mathcal D}}
\def\cH{{\mathcal H}}

\def\M{{\mathcal M}}
\def\cO{{\mathcal O}}
\def\X{{\mathcal X}}

\def\G{\Gamma}

\def\g{{\mathfrak g}}
\def\h{{\mathfrak h}}
\def\m{{\mathfrak m}}

\def\Dtilde{\widetilde{D}}
\def\Etilde{\widetilde{E}}

\def\Dbar{\overline{D}}
\def\Mbar{\overline{M}}

\def\Ybar{\overline{Y}}
\def\Zbar{\overline{Z}}

\def\hbar{\overline{h}}

\def\Ctilde{\widetilde{C}}

\def\Zhat{\widehat{\Z}}

\def\Xo{\overset{\circ}{X}}

\def\fbar{{\overline{f}}}
\def\ftilde{\tilde{f}}
\def\xtilde{\tilde{x}}

\def\Qp{{\Q_p}}

\def\Sp{{\mathrm{Sp}}}
\def\SL{{\mathrm{SL}}}
\def\GL{{\mathrm{GL}}}
\def\SO{{\mathrm{SO}}}
\def\SU{{\mathrm{SU}}}
\def\PSL{{\mathrm{PSL}}}

\def\hodge{\mathrm{Hodge}}
\def\red{\mathrm{red}}
\def\sing{\mathrm{sing}}
\def\div{\mathrm{div}}

\def\dot{\bullet}
\def\bs{\backslash}

\newcommand\id{\operatorname{id}} 
\newcommand\im{\operatorname{im}} 
\newcommand\Hom{\operatorname{Hom}}
\newcommand\Spec{\operatorname{Spec}}

\newcommand\Aut{\operatorname{Aut}}
\newcommand\Gr{\operatorname{Gr}}
\newcommand\Jac{\operatorname{Jac}}
\newcommand\Pic{\operatorname{Pic}}

\newcommand\Ext{\operatorname{Ext}}
\newcommand\Bl{\operatorname{Bl}}

\newcommand\rank{\operatorname{rank}}
\newcommand\trace{\operatorname{trace}}
\newcommand\codim{\operatorname{codim}}

\renewcommand\Im{\operatorname{Im}}

\newtheorem{theorem}{Theorem}[section]
\newtheorem{lemma}[theorem]{Lemma}
\newtheorem{proposition}[theorem]{Proposition}
\newtheorem{corollary}[theorem]{Corollary}
\newtheorem{bigtheorem}{Theorem}
\newtheorem{bigcorollary}[bigtheorem]{Corollary}

\theoremstyle{definition}
\newtheorem{definition}[theorem]{Definition}
\newtheorem{example}[theorem]{Example}

\theoremstyle{remark}

\begin{document}

\title{Monodromy of Codimension-One Sub-Families of Universal Curves}

\author{Richard Hain}
\address{Department of Mathematics\\ Duke University\\
Durham, NC 27708-0320}
\email{hain@math.duke.edu}

\thanks{Supported in part by grants DMS-0706955 and DMS-0103667 from the
National Science Foundation, and by MSRI}

\date{\today}


\begin{abstract}
Suppose that $g\ge 3$, that $n\ge 0$ and that $\ell\ge 1$. The main result is
that if $E$ is a smooth variety that dominates a codimension $1$ subvariety $D$
of $\M_{g,n}[\ell]$, the moduli space of $n$-pointed, genus $g$, smooth,
projective curves with a level $\ell$ structure, then the closure of the image
of the monodromy representation $\pi_1(E,e_o) \to \Sp_g(\Zhat)$ has finite index
in $\Sp_g(\Zhat)$. A similar result is proved for codimension $1$ families of
principally polarized abelian varieties.
\end{abstract}

\maketitle


\section{Introduction}

Suppose that $g,\ell,n$ are integers satisfying $g\ge 2$, $\ell\ge 1$ and $n\ge
0$. Denote the moduli space of smooth complex projective curves of genus $g$
with a level $\ell\ge 1$ structure by $\M_g[\ell]$ and the $n$th power of the
universal curve over it by $\cC_g^n[\ell]$.  The moduli space of smooth
$n$-pointed smooth projective curves of genus $g$ with a level $\ell\ge 1$
structure $\M_{g,n}[\ell]$ is a Zariski open subset of $\cC_g^n[\ell]$. These
will all be regarded as orbifolds. There is a natural monodromy representation
$$
\rho : \pi_1(\cC_g^n[\ell],x_o) \to \Sp_g(\Z)
$$
whose image is the level $\ell$-congruence subgroup of $\Sp_g(\Z)$.

The profinite completion of a discrete group $\G$ will be denoted by
$\G^\wedge$. Denote the profinite completion of the integers by $\Zhat$. A
homomorphism $\G \to \GL_N(\Z)$ induces a homomorphism $\G^\wedge \to
\GL_N(\Zhat)$.

\begin{bigtheorem}
\label{thm:weak-density}
Suppose that $g\ge 3$ and $n\ge 0$. If $E\to D$ is a dominant morphism from a
smooth variety to an irreducible divisor $D$ in $\cC_g^n[\ell]$, then the image
of the monodromy representation
$$
\pi_1(E,e_o)^\wedge \to \pi_1(D,d_o)^\wedge \to \Sp_g(\Zhat)
$$
has finite index in $\Sp_g(\Zhat)$.
\end{bigtheorem}

The statement is  false when $g=2$ as will be explained in
Section~\ref{sec:density}. We will prove a stronger version of this theorem
(Theorem~\ref{thm:density_curves}), in which $\cC_g^n[\ell]$ is replaced by a
``suitably generic linear section'' of dimension $\ge 3$ of it.  We also prove
similar result for abelian varieties (Theorem~\ref{thm:density_ppavs}).

Each rational representation $V$ of $\Sp_g$ determines a local system $\V$ over
$\cC_g^n[\ell]$. The theorem implies that when $V$ does not contain the trivial
representation, the low dimensional cohomology of $\cC_g^n[\ell]$ with
coefficients in $\V$ changes very little when $\cC_g^n[\ell]$ is replaced by a
Zariski open subset of $\cC_g^n[\ell]$ or by its generic point.

\begin{bigcorollary}
Suppose that $g\ge 3$, $n\ge 0$ and $\ell\ge 1$. If $U$ is a Zariski open subset
of $\cC_g^n[\ell]$, then for all non-trivial, irreducible representations $V$ of
$\Sp(H)$, the map
$$
H^j(\cC_g^n[\ell],\V) \to H^j(U,\V)
$$
induced by the inclusion $U\hookrightarrow \cC_g^n[\ell]$ is an isomorphism when
$j=0,1$ and an injection when $j=2$.
\end{bigcorollary}

The groups $H^j(\cC_g^n[\ell],V)$ are known for all $V$ when $j=1$ and $g\ge 3$;
a canonical subspace of it is known \cite{hain:torelli} when $j=2$ when $g\ge
6$. The resulting computation of the Galois cohomology groups
$H^j(G_{k(\cC_g^n[\ell])},V\otimes\Q_p)$ of the absolute Galois group of the
function field $k(\cC_g^n[\ell])$, when $k$ is a number field, is an essential
ingredient of the author's investigation \cite{hain:sec_conj} of rational points
of the universal curve over the function field of $\M_{g,n/k}[\ell]$.

The proof of Theorem~\ref{thm:weak-density} easily reduces to the case $n=0$.
Putman \cite{putman:pic} has proved that  the Picard group of $\M_g[\ell]$ has
rank $1$ when $g\ge 5$. A standard argument using the that fact that
$\M_g[\ell]$ is quasi-projective then implies that every irreducible divisor $D$
in $\M_g[\ell]$ is ample. A Lefschetz type theorem due to Goresky and MacPherson
\cite{smt} implies that when $D$ is generic, $\pi_1(D,d_o) \to
\pi_1(\M_g[\ell],x_o)$ is an isomorphism. In this case, the result is immediate.
The principal difficulty occurs when $D$ has self-intersections.
\begin{figure}[!ht]
\label{fig:normalization}
\epsfig{file=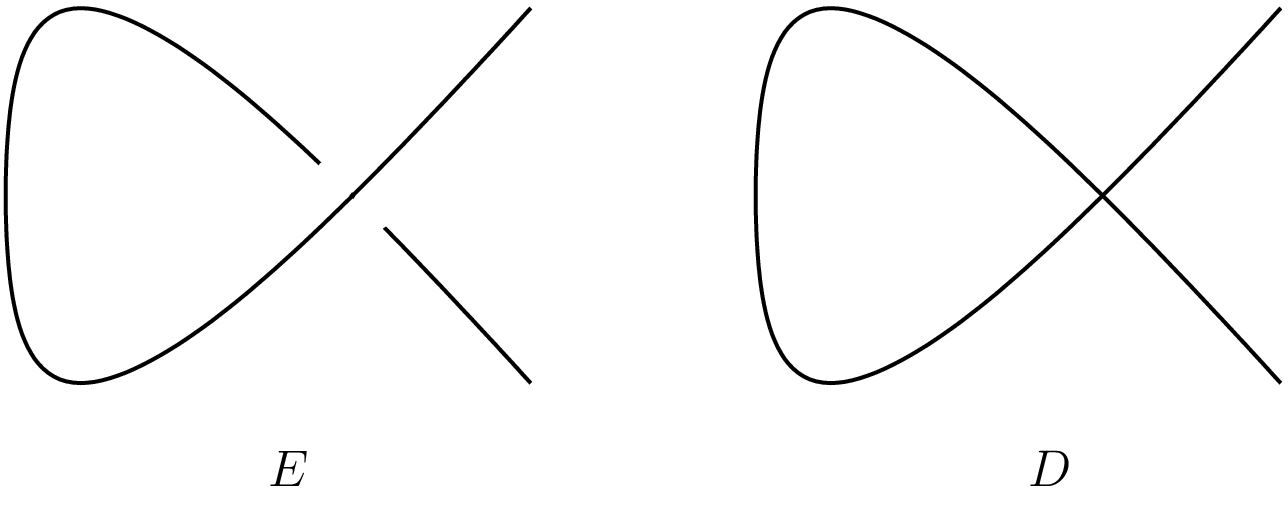, width=3in}
\caption{Normalization}
\end{figure}
In this case the image of the homomorphism $\pi_1(E,e_o) \to \pi_1(D,d_o)$
induced by its normalization $E\to D$ may have infinite index in $\pi_1(D,d_o)$,
as it does in Figure~\ref{fig:normalization}. This issue is addressed by a
technical result, Theorem~\ref{thm:strictness}, from which
Theorem~\ref{thm:weak-density} follows directly. It is a hybrid of a
``non-abelian strictness theorem'' and a Lefschetz-type theorem.

The term {\em non-abelian strictness theorem} needs further explanation. Input
for one such type of theorem is a diagram
\begin{equation}
\label{eqn:strict}
\xymatrix{
Z \ar[r]^f & Y \ar[r]^j & X
}
\end{equation}
of morphisms of varieties, where $X$ and $Z$ are smooth, and where $f$ is
dominant. Deligne proved several prototypical strictness theorems for cohomology
in \cite{deligne:hodge}. In the situation (\ref{eqn:strict}), a standard
strictness argument, given in Section~\ref{sec:discussion}, implies that the
image of $H_1(Z) \to H_1(X)$ has finite index in the image of $H_1(Y) \to
H_1(X)$, even though the image of $H_1(Z)$ may have infinite index in $H_1(Y)$.
We will refer to this as an ``abelian strictness theorem'' as the invariant
$H_1$ is the abelianization of the fundamental group. The most optimistic
formulation of a non-abelian strictness assertion would be that the image of
$\pi_1(Z,z_o)$ in $\pi_1(X,x_o)$ has finite index in the image of $\pi_1(Y,y_o)$
in $\pi_1(X,x_o)$. A weaker assertion would be that this holds for profinite
fundamental groups. A less optimistic statement would be that, for all reductive
linear representations $\pi_1(X,x_o) \to \GL_N(F)$, the image of
$\rho_Z:\pi_1(Z,z_o) \to \GL_N(F)$ has finite index in the image of
$\rho_Y:\pi_1(Y,z_o) \to \GL_N(F)$, or even that the Zariski closure of the
image of $\rho_Z$ has finite index in the Zariski closure of the image of
$\rho_Y$. All four assertions are false, as is shown by the example in
Section~\ref{sec:discussion}. This, I hope, goes part way towards justifying the
technical assumptions of Theorem~\ref{thm:strictness}.

\bigskip

\noindent{\em Conventions:}  An orbifold is a stack in the category of
topological spaces in which the automorphism group of each point is finite. All
the orbifolds considered in this paper are quotients of a contractible complex
manifold by a virtually free action of a finitely generated discrete group.
Equivalently, they are all quotients of a complex manifold by a finite group. A
detailed exposition of such orbifolds can be found in \cite{hain:elliptic}.

Unless stated to the contrary, all varieties are defined over $\C$. They will be
regarded as topological spaces (or orbifolds) in the complex topology.  A
divisor in a variety $X$ is a closed subvariety of pure codimension $1$.

Implicit in the statement that a map $f: X \to Y$ induces a homomorphism
$\pi_1(X,x_o) \to \pi_1(Y,y_o)$ is the assertion that $y_o = f(x_o)$.

\bigskip

\noindent{\em Acknowledgments:} Much of this work was done in 2009 during visits
to MSRI, the Universit\'e de Nice and the Newton Institute. I am very grateful
to these institutions for their hospitality and support and to Duke University
for the sabbatical leave during which this work was completed. I would like to
thank Tomohide Terasoma for pointing out a gap in an early proof of the main
theorem. Special thanks to the referee for suggesting improvements to the
exposition and for providing a proof of Proposition~\ref{prop:referee}, which
resolved a question in an earlier version of this paper and allowed a
superfluous assumption to be removed from Theorems~\ref{thm:density_curves} and
\ref{thm:density_ppavs}. He/she also suggested simplifications and improvements
to the constructions in Section~\ref{sec:ab_strictness}.

I am especially grateful to Madhav Nori who alerted me to a gap in the proof of
Lemma~\ref{lem:normalization} and contributed its current formulation and proof.
He also suggested the approach to proving Theorem~\ref{thm:weak-density} in
genus 4, pointing out the relevance of the remarkable result \cite{mvw} due
to Matthews, Vaserstein and Weisfeiler, which can also be found in
\cite[Thm.~5.4]{nori:finite}.

\section{Setup and Statement of the Strictness Theorem}
\label{sec:strictness}

\subsection{Setup}
\label{sec:setup}
The main objects in the Strictness Theorem are:
\begin{enumerate}

\item a smooth, connected quasi-projective variety $M$,
\label{item1}

\item a projective completion $\Mbar$ of $M$ and an open subvariety
$M^c$ that contains $M$,
\label{item2}

\item a local system $\V$ over $M^c$ whose fibers are torsion free $\Z$-modules.
\label{item3}

\end{enumerate}
We will suppose that:
\begin{enumerate}
\setcounter{enumi}{3}

\item each component of $\Mbar - M$ has codimension $\ge 2$ in $\Mbar$ and each
component of $\Mbar - M^c$ has codimension $\ge 3$ in $\Mbar$,
\label{item4}

\item $M^c$ is unibranch (i.e., locally analytically irreducible) at the generic
point of each codimension $2$ stratum of $M^c - M$;
\label{item5}

\end{enumerate}
Choose a basepoint $x_o$ of $M$. Denote the fiber of $\V$ over $f(x_o)$ by
$V_o$. Set $\G_o = \Aut V_o$. For each $m\ge 1$, set
$$
\G_o[m] = \{\phi \in \Aut V_o : \phi \equiv \id \bmod\ m\}.
$$
The inverse image of $\G_o[m]$ under the monodromy representation $\rho :
\pi_1(M,x_o) \to \G_o$ corresponds to a connected, unramified, Galois covering
$p_m : M_m \to M$.

\begin{enumerate}

\setcounter{enumi}{5}

\item $\Pic M_m$ is finitely generated of rank 1 for all $m\ge 1$.
\label{item6}

\end{enumerate}
Note that if $\dim M = 1$, then $M=\P^1$.

\subsection{An example}
The motivating example, to be explained in detail in Section~\ref{sec:satake},
is:
\begin{itemize}

\item $M$ is $\M_g[\ell]$, the moduli space of smooth projective curves of genus
$g\ge 5$ with a level $\ell \ge 3$ structure;

\item $\Mbar$ is its Satake compactification and $M^c$ is the open subset whose
points correspond to genus $g$ curves of compact type;

\item $\V$ is the local system over $M^c$ whose fiber over the moduli point of
the curve $C$ is the first cohomology $H^1(\Jac C;\Z)$ of its jacobian.

\end{itemize}
In this example $M_m = \M_g[m']$ where $m'=\mathrm{lcm}(m,\ell)$.

Another example is provided by taking $M=M^c$ to be the moduli space
$\A_g[\ell]$ of principally polarized abelian varieties with a level $\ell \ge
3$ structure, where $g\ge 3$. Details will be given in Section~\ref{sec:satake}.

\subsection{Statement of the Strictness Theorem}

\begin{theorem}
\label{thm:strictness}
Suppose that $D$ is an irreducible divisor in $M$ and that $h : E \to D$ is a
dominant morphism. If the conditions in Section~\ref{sec:setup} are satisfied,
then the image of $\pi_1(E,e_o)^\wedge \to \G_o^\wedge$ has finite index in the
image of $\pi_1(M,x_o)^\wedge \to \G_o^\wedge$.
\end{theorem}

It is not clear that all of these hypotheses are necessary. Inspired by the
abelian case (Proposition~\ref{prop:abelian}), one might suspect that the image
of $\pi_1(E,e_o)^\wedge \to \pi_1(M,x_o)^\wedge$ has finite index in the image
of $\pi_1(D,x_o)^\wedge \to \pi_1(M,x_o)^\wedge$. Unfortunately, this is not the
case as we show in Section~\ref{sec:examples}, where we present examples that
suggest that most of the hypotheses above are necessary.

\section{Topological and Geometric Preliminaries}
\label{sec:lefschetz}

\subsection{Topological preliminaries}

One of the main tools used to prove the Strictness Theorem is a Lefschetz type
theorem proved by Goresky and MacPherson in their book \cite{smt} on stratified
Morse theory. For the convenience of the reader, we state the special case of
this theorem that we will be using repeatedly.

Fix a real analytic riemannian metric on $\P^N$. For a subset $A$ of $\P^N$ and
a real number $\delta > 0$, set
$$
A_\delta :=\{ x \in \P^N : \mathrm{dist}(x,A) < \delta \}.
$$

\begin{theorem}[{Goresky and MacPherson \cite[p.~150]{smt}}]
\label{thm:smt}
Suppose that $X$ is an $n$-dimensional, connected complex algebraic manifold and
that $f : X \to \P^N$ is a morphism, all of whose fibers are finite. If $L$ is a
codimension $c$ linear subspace of $\P^N$, then for all sufficiently small
$\delta>0$, the homomorphism
\begin{equation}
\label{eqn:smt}
\pi_j(f^{-1}(L_\delta), x) \to \pi_j(X,x)
\end{equation}
induced by the inclusion is an isomorphism when $j<n(c)$ and surjective when
$j=n(c)$, where $n(c) = \max\{n-c,-1\}$. If $L$ is generic or if $f^{-1}(L)$ is
proper, then $L_\delta$ may be replaced by $L$ in (\ref{eqn:smt}).
\end{theorem}

The final statement is slightly stronger than that made on page 151 of
\cite{smt}. The stronger statement follows from a result of Durfee \cite{durfee}
which implies that when $f^{-1}(L)$ is compact, $f^{-1}(L_\delta)$ is a regular
neighbourhood of $f^{-1}(L)$, so that the inclusion $f^{-1}(L)\hookrightarrow
f^{-1}(L_\delta)$ is a homotopy equivalence.

The theorem implies that generic linear sections of $X$ of dimension $\ge 3$
have the same Picard group.

\begin{corollary}
Assume the notation of the theorem above.
If $L$ is a generic linear subspace of $\P^N$ of codimension $\le \dim X - 3$,
then the inclusion $f^{-1}(L) \to X$ induces an isomorphism
$$
\Pic X \to \Pic f^{-1}(L).
$$
\end{corollary}

\begin{proof}
Since $X$ and $f^{-1}(L)$ are smooth, the assertion follows from
Theorem~\ref{thm:smt} and the fact that for a complex algebraic manifold $Y$ the
sequence
\begin{equation}
\label{eqn:pic}
0 \to \Ext^1_\hodge(\Z,H^1(Y,\Z(1))) \to \Pic Y \to
\Hom_\hodge(\Z(0),H^2(Y,\Z(1))) \to 0
\end{equation}
is exact.
\end{proof}

A direct consequence is that the hypotheses of Theorem~\ref{thm:strictness}
behave well when taking generic linear sections of dimension $\ge 3$.

\begin{corollary}
\label{cor:linear}
Suppose that the members of the diagram
$$
M \hookrightarrow M^c \hookrightarrow \Mbar
$$
and the local system $\V$ over $M^c$ satisfy the conditions in
Section~\ref{sec:setup}. If $\Mbar \to \P^N$ is a projective imbedding and $L$
is a generic linear subspace of $\P^N$ such that $\dim(M\cap L)\ge 3$, then the
members of the diagram
$$
M\cap L \hookrightarrow M^c \cap L \hookrightarrow \Mbar \cap L
$$
and the restriction of $\V$ to $M^c\cap L$ also satisfy the conditions in
Section~\ref{sec:setup}.
\end{corollary}

The following fact will be used in the proof the the Strictness Theorem.

\begin{corollary}
\label{cor:ample}
If the members of
$$
M \hookrightarrow M^c \hookrightarrow \Mbar
$$
satisfy the conditions in Section~\ref{sec:setup}, then $M$ contains a smooth
complete curve and every irreducible divisor in $M$ is ample.
\end{corollary}

\begin{proof}
Fix a projective imbedding $f:\Mbar \hookrightarrow \P^N$. Since each component
of $\Mbar-M$ has codimension $\ge 2$, the intersection $M\cap L$ of a generic
linear subspace $L$ of $\P^N$ of codimension $\dim M - 1$ is a complete curve in
$M$. If $D$ is an irreducible divisor in $M$, then the image of its closure
$f(\Dbar)$ will have degree $L\cdot \fbar_\ast\Dbar$ in $\P^N$. This degree is
positive. Since $\Pic M$ is finitely generated of rank 1, $D$ is ample in $M$.
\end{proof}

\subsection{Dominant morphisms}

This section contains a few elementary results that are needed in the proof
of the strictness theorem.

The link of every stratum of a complex algebraic variety is connected if and
only if $X$ is unibranch. From this it follows that:

\begin{lemma}
\label{lem:surjective}
If $U$ is a Zariski open subset of an irreducible, unibranch variety
$X$, then the induced homomorphism $\pi_1(U,x_o) \to \pi_1(X,x_o)$ is
surjective. \qed
\end{lemma}

The following is standard. Cf.\ \cite[p.~363]{bourbaki}.

\begin{lemma}
\label{lem:unibranch}
Every normal variety is unibranch. \qed
\end{lemma}

Suppose $D$ is an irreducible algebraic variety. Denote its normalization by
$\Dtilde \to D$. 

\begin{proposition}
\label{prop:normalization}
If $f:E\to D$ is a dominant morphism from a variety to $D$, then the image of
$\pi_1(E,e_o) \to \pi_1(D,x_o)$ contains a finite index subgroup of the image of
$\pi_1(\Dtilde,x_o) \to \pi_1(D,x_o)$. Here $e_o \in E$ is chosen so that
$x_o=f(e_o)$ is a smooth point of $D$.
\end{proposition}

\begin{proof}
Let $\Etilde \to E$ be the normalization of $E$. The normalization of $D$ in the
function field of $\Etilde$ is its normalization $\Dtilde$. Since the diagram
$$
\xymatrix{
\Etilde \ar[r]\ar[d]_\ftilde & E \ar[d]^f \cr
\Dtilde \ar[r] & D
}
$$
commutes, it suffices to show that the image of the homomorphism
$\pi_1(\Etilde,e_o) \to \pi_1(\Dtilde,x_o)$ has finite index in
$\pi_1(\Dtilde,x_o)$. There is a smooth Zariski open subset $U$ of $\Dtilde$ and
a smooth Zariski open subset $V$ of $\ftilde^{-1}(U)$ such that the restriction
of $\ftilde$ to $V$ is a locally trivial fiber bundle. Since the number of
connected components of each fiber is finite, the image of $\pi_1(V,e_o) \to
\pi_1(U,x_o)$ has finite index in $\pi_1(U,x_o)$. The result now follows as the
homomorphisms $\pi_1(V,e_o) \to \pi_1(\Etilde,e_o)$ and $\pi_1(U,x_o) \to
\pi_1(\Dtilde)$ are surjective by Lemma~\ref{lem:unibranch} and
Lemma~\ref{lem:surjective}.
\end{proof}

Specializing to the case where $D$ is smooth, we obtain the following special
case.

\begin{corollary}
\label{cor:reduction}
Suppose that $Y$ is a smooth variety and that $f : X \to Y$ is a dominant
morphism. If $E$ is an irreducible divisor in $X$ and if the restriction of $f$
to $E$ is dominant, then the image of $\pi_1(E,e_o) \to \pi_1(Y,y_o)$ has finite
index in $\pi_1(Y,y_o)$.
\end{corollary}

\subsection{Lean morphisms}

The Strictness Theorem applies to more general situations than those described
in Section~\ref{sec:strictness}. In order to describe them, we will need the
following definition. Recall that, by a divisor in an irreducible variety $X$,
we mean a closed subvariety of pure codimension $1$.

\begin{definition}
\label{def:lean}
A dominant morphism $F : X \to Y$ is {\em lean} if the codimension of the
inverse image $F^{-1}(Z)$ of an irreducible closed subvariety $Z$ of $Y$ is a
divisor in $X$ implies that $Z$ is a divisor in $Y$.
\end{definition}

Every dominant morphism $F : X \to Y$ all of whose closed fibers have the same
dimension, is lean. In particular, flat morphisms are lean.

\begin{lemma}
\label{lem:families}
Suppose that the strictness theorem (Theorem~\ref{thm:strictness}) holds for the
local system $\V$ over $M$. If $h : X \to M$ is a lean morphism, then the
strictness theorem holds for the pullback of $\V$ to $X$. That is, if $D$ is an
irreducible divisor in $X$, and if $E \to D$ is a dominant morphism from a
smooth variety to $D$, then the image of $\pi_1(E,e_o)^\wedge \to \G_o^\wedge$
has finite index in the image of $\pi_1(M,x_o)^\wedge \to \G_o^\wedge$.
\end{lemma}

\begin{proof}
Since $h$ is lean, the image of $D$ in $M$ has codimension $\le 1$. When $D \to
M$ is dominant, so is $E\to M$. Corollary~\ref{cor:reduction} implies that the
the image of $\pi_1(E,e_o)$ in $\pi_1(M,x_o)$ has finite index in
$\pi_1(M,x_o)$, which proves the result in this case. If the closure $F$ of the
image of $D$ in $M$ is a divisor, then $E\to F$ is dominant so that the
hypothesis that the strictness theorem holds for $M$ implies that the image of
$$
\pi_1(E,e_o)^\wedge \to \pi_1(M,x_o)^\wedge \to \G_o^\wedge
$$
has finite index in the image of $\pi_1(M,x_o)^\wedge \to \G_o^\wedge$.
\end{proof}

The proof of the following proposition and the lemma were contributed by the
referee in response to a question in an earlier version of this paper. They
guarantee that generic hypersurface sections of a flat family of varieties are
lean.

\begin{lemma}
If $Y$ is an algebraic subvariety of $\P^N$ of positive dimension $d$, then the
locus in $\P(H^0(\P^N,\cO(e)))$ of hypersurfaces of degree $e$ that contain $Y$
has codimension at least $e^d/d!$.
\end{lemma}

\begin{proof}
Replace $Y$ by $Y^\red$ if necessary. Let $y$ be a smooth point of $Y$. Since
$\dim Y = d$,
$$
\dim H^0\big(\Spec\cO_{Y,y}/\m_{Y,y}^{e+1},\cO(e)\big)
= \binom{d+e}{e} \ge e^d/d!
$$
The result follows as the composite of the maps
$$
\xymatrix@C=18pt{
H^0(\P^N,\cO(e)) \ar[r]^(0.35)\simeq &
H^0\big(\Spec\cO_{\P^N\!,y}/\m_{\P^N\! ,y}^{e+1},\cO(e)\big) \ar[r] &
H^0\big(\Spec\cO_{Y,y}/\m_{Y,y}^{e+1},\cO(e)\big)
}
$$
is surjective.
\end{proof}

\begin{proposition}
\label{prop:referee}
Suppose that $X$ and $T$ are quasi-projective varieties. Fix an imbedding $X\to
\P^N$. If $f:X\to T$ is a dominant morphism, all of whose fibers have the same
dimension $d>0$, then there exists $e>0$ such that for a generic hypersurface
$Y$ of $\P^N$ of degree $e$, the restriction $f_S : S \to T$ of $f$ to $S:=
Y\cap X$ is a dominant morphism each of whose fibers has dimension $d-1$. In
particular, $f_S$ is lean.
\end{proposition}

\begin{proof}
For $e>0$, let $Z$ be the subvariety of $\P\big(H^0(\P^N,\cO(e))\big)\times T$
consisting of pairs $(Y,t)$ where the degree $d$ hypersurface $Y$ contains the
fiber of $f$ over $t\in T$. The lemma above implies that $Z$ has codimension
$\ge e^d/d!$. Denote by $W$ the Zariski closure of the image of $Z$ under the
projection of $\P\big(H^0(\P^N,\cO(e))\big)\times T$ onto its first factor. This
has codimension $\ge e^d/d! - \dim T$. If we choose $e$ such that $e^d/d! > \dim
T$, then $W$ will be a proper subvariety of $\P\big(H^0(\P^N,\cO(e))\big)$,
which implies the result.
\end{proof}

By taking iterated generic hypersurface sections, we conclude:

\begin{corollary}
\label{cor:generic_lean}
Suppose that $X$ and $T$ are quasi-projective varieties. Fix an imbedding $X\to
\P^N$. If $f:X\to T$ is a dominant morphism, all of whose fibers have the same
dimension $d>0$, then for a section of $X$ by a generic complete intersection of
codimension $c \le \codim T$ of multi-degree $(e_1,\dots,e_c)$, where $0 \ll
e_1\ll e_2 \ll \dots \ll e_c$, the restriction $f_S : S \to T$ of $f$ to $S:=
L\cap X$ is a dominant morphism each of whose fibers has dimension $d-c$. In
particular, $f_S$ is lean.
\end{corollary}

\section{Moduli spaces and their Baily-Borel-Satake Compactifications}
\label{sec:satake}

Suppose that $g\ge 1$ and that $\ell\ge 1$. Following work of Satake
\cite{satake}, Baily and Borel \cite{bb} constructed a minimal compactification
$X^\ast$ of each locally symmetric variety $X$. We will call $X^\ast$ the {\em
Satake compactification} of $X$.   It is a normal complex projective algebraic
variety. The Satake compactification $\A_g^\ast[\ell]$ of $\A_g[\ell]$ has the
property that its boundary $\A^\ast_g[\ell]-\A_g[\ell]$ has complex codimension
$g$.

The only imbeddings of $\A_g[\ell]$ into projective space that we will consider
are ones that extend to imbeddings of $\A^\ast_g[\ell]$. All such imbeddings are
given by automorphic forms of a sufficiently high weight. By a linear section of
$\A_g[\ell]$, we mean a linear section with respect to such an imbedding.

\begin{proposition}
\label{prop:M-satake}
If $g\ge 2$ and $\ell \ge 3$, then $\M_g[\ell]$ has a projective completion
$\M^\ast_g[\ell]$ such that the period mapping $f : \M_g[\ell] \to \A_g[\ell]$
extends to a finite morphism
$$
\fbar : \M^\ast_g[\ell] \to \A^\ast_g[\ell].
$$
When $g\ge 3$, the subvariety $\fbar^{-1}(\A^\ast_g[\ell]-\A_g[\ell])$ of
$\M^\ast_g[\ell]$ has codimension $3$ and $\M_g^\ast[\ell]$ is unibranch at the
generic point of each codimension $2$ stratum of $\M_g^\ast[\ell]-\M_g[\ell]$.

\end{proposition}

\begin{proof}
Define $\M_g^\ast[\ell]$ to be the normalization of $\A_g^\ast[\ell]$ in the
function field of $\M_g[\ell]$. Since $\M_g[\ell]$ is smooth (and thus normal)
when $\ell\ge 3$, $\M_g[\ell]$ is a subvariety of $\M_g^\ast[\ell]$. The
projection $\fbar : \M_g^\ast[\ell]\to\A_g^\ast[\ell]$ is finite.  The generic
point of a boundary component of $\A_g^\ast[\ell]$ is a principally polarized
abelian variety of dimension $g-1$ with a level $\ell$ structure. The generic
point of the image of the boundary $\M_g^\ast[\ell]-\M_g[\ell]$ is the jacobian
of a smooth projective curve of genus $g-1$, which has dimension $3g-6$. Since
$\fbar$ is finite, this implies that $\fbar^{-1}(\A^\ast_g[\ell]-\A_g[\ell])$
has codimension $3$ in $\M_g^\ast[\ell]$.

The moduli space $\M_g^c$ of genus $g$ complex projective curves of compact type
is the complement of the divisor $\Delta_0$ in the Deligne-Mumford
compactification $\overline{\M}_g$ of $\M_g$. It will be regarded as an
orbifold. The period mapping extends to a proper mapping $\M_g^c \to \A_g$.
Since
$$
\pi_1(\M_g,x_o) \to \pi_1(\M_g^c,x_o) \to \pi_1(\A_g,a_o) \cong \Sp_g(\Z)
$$
is surjective, so is $\pi_1(\M_g^c,x_o) \to \pi_1(\A_g,a_o) \cong \Sp_g(\Z)$.
Denote the Galois covering of $\M_g^c$ corresponding to the kernel of the
homomorphism $\pi_1(\M_g^c,x_o) \to \Sp_g(\Z/\ell\Z)$ by $\M_g^c[\ell]$. This is
a smooth orbifold that contains $\M_g[\ell]$ as a Zariski open subset. Since
$\M_g^c[\ell]$ is normal, the period mapping
$$
\M_g^c[\ell] \to \A_g[\ell] \hookrightarrow \A_g^\ast[\ell]
$$
factors through $\M_g^\ast[\ell] \to \A^\ast[\ell]$. The fiber of $\M_g^c[\ell]
\to \M_g^\ast[\ell]$ over the generic point of a codimension $2$ stratum of
$\M_g^\ast[\ell]-\M_g[\ell]$ is a smooth projective curve of genus $g-1$. The
link of the image of such a genus $g-1$ curve $C$ in $\M_g^\ast[\ell]$ is the
link of $C$ in $\M_g^c[\ell]$, which is the unit tangent bundle of $C$ and is
therefore connected. It follows that $\M_g^c[\ell]$ is unibranch at the
generic point of each codimension 2 component of $\M_g^\ast[\ell]-\M_g[\ell]$.
\end{proof}

\section{Monodromy Theorems}
\label{sec:density}

There are two significant classes of examples where the hypotheses of
Theorem~\ref{thm:strictness} are satisfied and in which we can prove
Theorem~\ref{thm:weak-density} and generalizations. These are constructed from
moduli spaces of curves, and from moduli spaces of abelian varieties. In both
cases the local system $\V$ over $M$ is pulled back from the local system over
$\A_g$ whose fiber over the moduli point of the abelian variety $A$ is
$H^1(A;\Z)$. In both cases the image of the monodromy homomorphism is a finite
index subgroup of
$$
\Aut(H^1(A,\Z),\theta)
$$
where $\theta : H^1(A,\Z)^{\otimes 2} \to \Z$ denotes the polarization. The
choice of a symplectic basis of $H^1(A;\Z)$ gives an isomorphism of this group
with $\Sp_g(\Z)$, the group of $2g\times 2g$ integral symplectic matrices.

\subsection{Curves}

Suppose that $g\ge 5$ and that $\ell \ge 3$. Then $\M_g[\ell]$ is a smooth
smooth quasi-projective variety. Fix an imbedding $\M_g^\ast[\ell]
\hookrightarrow \P^N$ of the Satake compactification of $\M_g[\ell]$. Suppose
that $L$ is a linear subspace of $\P^N$ that is generic with respect to the
inclusion $f : \M_g^\ast[\ell] \hookrightarrow \P^N$. and that satisfies $\dim
(L\cap \M_g[\ell]) \ge 3$. Set $M = f^{-1}(L)$. This is a smooth subvariety of
$\M_g[\ell]$ of dimension $\ge 3$.

Suppose that $n\ge 0$. Denote the restriction of the universal curve over
$\M_g[\ell]$ to $M$ by $\cC_M$ and its $n$th power by $\cC_M^n$. This is a
quasi-projective variety. Fix an imbedding $\cC_M^n \hookrightarrow \P^r$.

\begin{theorem}
\label{thm:density_curves}
If $X$ is a generic section of $\cC^N_M$ in $\P^r$ by a complete intersection of
codimension $c\le \codim M$ and multi-degree $(e_1,\dots,e_c)$, where $0\ll e_1
\ll e_2 \ll \dots \ll e_c$, then
\begin{enumerate}

\item $X$ is smooth,

\item if $D$ is an irreducible divisor in $X$ and $E \to D$ is a dominant
morphism from a smooth variety to $D$, then the image of $\pi_1(E,e_o)^\wedge
\to \Sp_g(\Zhat)$ is a finite index subgroup of $\Sp_g(\Zhat)$.

\end{enumerate}
\end{theorem}

\begin{proof}
Proposition~\ref{prop:M-satake} and Putman's computation of the Picard groups of
the $\M_g[m]$ implies that $\V$ and the members of
$$
\M_g[\ell] \hookrightarrow \M_g[\ell]^c \hookrightarrow \M_g^\ast[\ell]
$$
satisfy the hypotheses of Theorem~\ref{thm:strictness}.
Corollary~\ref{cor:linear} implies that for generic linear subspaces $L$ of
$\P^N$ such that $\dim(L\cap \M_g[\ell])\ge 3$, the members of
$$
\M_g[\ell]\cap L \hookrightarrow \M_g[\ell]^c\cap L
\hookrightarrow \M_g^\ast[\ell] \cap L
$$
also satisfy the hypotheses of Theorem~\ref{thm:strictness}. We will assume that
$L$ is such a subspace. Set $M=\M_g[\ell]$, $M^c=\M_g^c[\ell]$ and $\Mbar =
\M_g^\ast[\ell]$. Theorem~\ref{thm:smt} implies that the inclusion
$M\hookrightarrow \M_g[\ell]$ induces an isomorphism on fundamental groups, so
that the image of $\pi_1(M,x_o) \to \Sp_g(\Z)$ is $\Sp_g(\Z)[\ell]$.

Corollary~\ref{cor:generic_lean} implies that $X\to M$ is lean.
Lemma~\ref{lem:families} implies that the image of $\pi_1(E,e_o)^\wedge$ in
$\Sp_g(\Zhat)$ is a finite index subgroup.
\end{proof}

\subsection{Abelian varieties}

Suppose that $g\ge 3$ and that $\ell \ge 3$. Then $\A_g[\ell]$
is a smooth smooth quasi-projective variety. Fix any imbedding
$$
\A_g^\ast[\ell] \hookrightarrow \P^N
$$
of the Satake compactification of $\A_g[\ell]$ given by automorphic forms.
Suppose that $L$ is a linear subspace of $\P^N$ that is generic with respect to
the imbedding $\A_g^\ast[\ell]\hookrightarrow \P^N$ and that satisfies $\dim
(L\cap \A_g^\ast[\ell]) \ge 3$. Set $M = L\cap \A_g[\ell]$. This is a smooth
subvariety of $\A_g[\ell]$ of dimension $\ge 3$.

Suppose that $n\ge 0$. Denote the restriction of the universal abelian variety
over $\A_g[m]$ to $M$ by $\X_M$ and its $n$th power by $\X_M^n$. This is a
quasi-projective variety. Fix an imbedding $\X_M^n \hookrightarrow \P^r$.

\begin{theorem}
\label{thm:density_ppavs}
If $X$ is a generic section of $\X^N_M$ in $\P^r$ by a complete intersection of
codimension $c\le \codim M$ and multi-degree $(e_1,\dots,e_c)$, where $0\ll e_1
\ll e_2 \ll \dots \ll e_c$, then
\begin{enumerate}

\item $X$ is smooth,

\item if $D$ is an irreducible divisor in $X$ and $E \to D$ is a dominant
morphism from a smooth variety to $D$, then the image of $\pi_1(E,e_o)^\wedge
\to \Sp_g(\Zhat)$ is a finite index subgroup of $\Sp_g(\Zhat)$.

\end{enumerate}
\end{theorem}

The assumptions are satisfied when $X=\X_M^N$. The case where $M=\A_g[\ell]$ is
an analogue of Theorem~\ref{thm:weak-density} for principally polarized abelian
varieties.

\begin{proof}
Borel's computation of the stable cohomology of arithmetic groups
\cite{borel:coho} (see also \cite[Thm.~3.2]{hain:torelli}) implies that for all
$g\ge 3$ and $\ell \ge 1$, $H^1(\A_g[\ell],\Z)$ vanishes and
$H^2(\A_g[\ell],\Z)$ is finitely generated of rank 1. The exact sequence
(\ref{eqn:pic}) now implies that $\Pic A_g[m]$ is finitely generated of rank
$\le 1$ for all $m\ge \ell$. But since the the determinant of the Hodge bundle
has non-trivial Chern class, it follows that $\Pic\A_g[m]$ has rank 1 for all
$m\ge \ell$. This implies that the members of
$$
\A_g[\ell] = \A_g^c[\ell] \hookrightarrow \A_g^\ast[\ell]
$$
satisfy the hypotheses of Theorem~\ref{thm:strictness}. The rest of the proof is
almost identical with that of Theorem~\ref{thm:density_curves} and is left to
the reader.
\end{proof}

\subsection{Proof of Theorem~\ref{thm:weak-density}}

When $g\ge 5$ and $\ell \ge 3$, this follows from
Theorem~\ref{thm:density_curves} by taking $M=\M_g[\ell]$ and $X=\cC_g^n[\ell]$.
The cases where $\ell = 1,2$ are both immediate consequences of the case $\ell =
4$.

Suppose that $g=3$ and that $D$ is an irreducible divisor in $\M_3[\ell]$. Since
the period mapping $f:\M_g[\ell] \to \A_3[\ell]$ is quasi-finite and dominant,
the Zariski closure $\Dbar$ of the image of $D$ in $\A_3[\ell]$ is a divisor in
$\A_3[\ell]$. If $E \to D$ is dominant, then $E \to \Dbar$ is also dominant.
Theorem~\ref{thm:density_ppavs} now implies that when $E$ is smooth, the image
of $\pi_1(E,e_o)^\wedge$ in $\Sp_3(\Zhat)$ has finite index.

It remains to prove the genus $4$ case. Here we take a different approach
suggested by Nori. It suffices to prove the result when $\ell \ge 3$. Let $G$ be
the Zariski closure of the image of $\pi_1(E,e_o)$ in the $\Q$-group
$$
\Aut(H_1(C_o,\Q),\theta) \cong \Sp_{4/\Q},
$$
where $C_o$ denotes the curve corresponding to the point $x_o \in \M_4$ and
$\theta$ its polarization. Standard arguments imply that $G$ is of Hodge type.
In particular, its Lie algebra $\g$ is a sub Hodge structure of the Lie algebra
$S^2 H_1(C_o)$ of $\Aut(H_1(C_o,\Q),\theta)$. The associated symmetric space $\D
:= G(\R)/K_G$ is hermitian as its tangent space $\g_\C/F^0\g \cong \g^{-1,1}$ at
the base point is a complex subspace of the tangent space of $\h_g$, the
symmetric space of $\Sp_4(\R)$. Set $\G = \Sp_4(\Z)[\ell]\cap G(\R)$. Then
$X:=\G\bs\D$ is a locally symmetric subvariety of $\A_4[\ell]$. The theorem of
the fixed part implies that $X$ contains the image of $D$ in $\A_4[\ell]$ from
which it follows that $\dim X \ge \dim D = 8$.

Write $\D$ as a product $\prod_{j=1}^N \D_j$ of irreducible, symmetric spaces
$\D_j = G_j(\R)/K_j$, where $G_j$ is a real Lie group and $K_j$ is a maximal
compact subgroup. Each $\D_j$ is hermitian, \cite[p.~381]{helgason}. The
classification of hermitian symmetric spaces \cite[p.~518]{helgason} implies
that $G_j$ is isogenous to $\SU(p,q)$ with $p+q \le 5$, to $\SO_o(n,2)$ with
$n\le 7$, to $\Sp_n(\R)$ with $n\le 4$, or to $\SO^\ast(2n)$ with $n\le 4$.

\begin{table}
\begin{tabularx}{4.5in}{|c||C|C|C|C|}
\hline
$G$ & $\SU(p,q)$ & $\SO_o(n,2)$ & $\Sp_n(\R)$ & $\SO^\ast(2n)$ \cr
\hline
$\dim_\C G/K$ & $pq$ & $n$ & $n(n+1)/2$ & $n$ \cr
$\rank_\C G_\C$ & $p+q-1$ & $1+\lfloor n/2 \rfloor$ & $n$ & $n(n-1)/2$ \cr
\hline
\end{tabularx}
\bigskip
\caption{}
\label{table:data}
\end{table}

Let $r_j = \rank_\C G_j$, the rank of the complexification of the Lie algebra of
$G_j$, and let $d_j = \dim_\C \D_j$. Since $\prod G_j \subseteq \Sp_4$ and since
$\prod \D_j \subseteq \h_4$, we have
$$
r_1 + \dots + r_N \le \rank_\C \Sp_4 = 4 \text{ and }
8 \le d_1 + \dots + d_N = \dim\D \le 10.
$$
The information in Table~\ref{table:data} implies that the only solution is
$N=1$ and $G=\Sp_4(\R)$, which forces $X$ to be $\A_4[\ell]$. The main result of
\cite{mvw} (see also \cite[Thm.~5.4]{nori:finite}) now implies that the closure
of the image of $\pi_1(E,e_o)$ in $\Sp_4(\Zhat)$ has finite index.

\subsection{Genus $2$}
\label{sec:genus2}

As remarked in the introduction, Theorem~\ref{thm:weak-density} (and hence
Theorem~\ref{thm:density_curves}) does not hold when $g=2$.
Theorem~\ref{thm:density_ppavs} also fails when $g=2$.

Fix $\ell \ge 3$. Denote the Siegel upper half space of rank $g$ by
$\h_g$:
$$
\h_g = \{Z \in \M_g(\C) : Z=Z^T \text{ and } \Im Z > 0\}.
$$
This is the symmetric space of $\Sp_g(\R)$. When $g=1$, it is the usual upper
half plane $\{\tau \in \C : \Im\tau > 0\}$. For each $\ell \ge 1$ we have
$$
\A_g[\ell] = \Sp_g(\Z)[\ell] \bs \h_g.
$$

View $\h_1\times\h_1$ as a submanifold of $\h_2$ via the inclusion
$$
(\tau_1,\tau_2)\mapsto
\begin{pmatrix}
\tau_1 & 0 \cr
0 & \tau_2
\end{pmatrix}
.
$$
The locus in $\h_2$ of (framed) principally polarized abelian surfaces that
are the product (as a polarized abelian variety) of two elliptic curves is
$$
\h_2^\red = \bigcup_{\gamma \in \Sp_2(\Z)} \gamma(\h_1\times \h_1).
$$
For each $\ell\ge 1$, the locus of reducible abelian surfaces $\A_2^\red[\ell]$
is the image in $\A_2[\ell]$ of $\h_2^\red$. When $\ell \ge 3$ the period
mapping induces an isomorphism\footnote{If interpreted in the category of
orbifolds, this is true for all $\ell\ge 1$.}
$$
\M_2[\ell] \cong \A_2[\ell]-\A_2[\ell]^\red.
$$
This is well-known --- cf.\ \cite[Prop.~6]{hain:ab_var}.

Choose $g \in \Sp_2(\Q)$ so that $g(\h_1\times \h_1) \not\subseteq \h_2^\red$.
This is possible because $\Sp_2(\R)$ acts transitively on $\h_2$ and because
$\Sp_2(\Q)$ is dense in $\Sp_2(\R)$ in the classical topology. Set
$$
\G = \big(\SL_2(\Z)\times\SL_2(\Z)\big) \cap g^{-1}\Sp_2(\Z)[\ell]g.
$$
This is an arithmetic subgroup of $\SL_2(\Q)^2$. Set $E^c = \G\bs(\h_1\times
\h_1)$ and define $f : E^c \to \A_2[\ell]$ to be the morphism induced by
$$
(\tau_1,\tau_2) \mapsto
g\begin{pmatrix}
\tau_1 & 0 \cr
0 & \tau_2
\end{pmatrix}
.
$$
When $\ell \ge 3$, $\G$ is torsion free and $E^c$ is smooth. Set $E =
f^{-1}(\M_2[\ell])$. This is a Zariski open subset of $E^c$. When $\ell \ge 3$,
the homomorphism $\pi_1(E,e_o) \to \pi_1(E^c,e_o)$ is surjective. The Zariski
closure of the image of $\pi_1(E,e_o) \to \pi_1(\A_2[\ell]$ is $g\G g^{-1}$,
which has Zariski closure $g \SL_2(\Q)^2 g^{-1}$ in $\Sp_2(\Q)$.

Let
$$
D^c = \text{ image of } E^c \text{ in } \A_2[\ell].
$$
This is a codimension one algebraic subvariety of $\A_2[\ell]$. Set $D=D^c \cap
\M_2[\ell]$. The choice of $g$ guarantees that $D$ is non-empty.

When $\ell \ge 3$, the diagram $E \to D \hookrightarrow \M_2[\ell]$ is a counter
example to Theorem~\ref{thm:weak-density} and the diagram $E^c \to D^c
\hookrightarrow \A_2[\ell]$ is a counter example to
Theorem~\ref{thm:density_ppavs} when $g=2$ as the image of $\pi_1(E,e_o) \to
\Sp_2(\Q)$ is not Zariski dense.

\section{Cohomological Applications}
\label{sec:applications}

Suppose that the members of the diagram
$$
M \hookrightarrow M^c \hookrightarrow \Mbar
$$
and the local system $\V$ over $M^c$ satisfy the conditions in
Section~\ref{sec:setup}. Fix a prime number $p$. Denote the Zariski closure
of the image of the monodromy homomorphism $\pi_1(M,x_o) \to
\Aut(V_o\otimes\Qp)$ by $R$. This is a $\Qp$-group. Denote by $\W$ the local
system of $\Qp$ vector spaces over $M$ that corresponds to the $R$-module $W$.

\begin{theorem}
Suppose that $R$ is a connected reductive group and that $X\to M$ is a lean
morphism. If $U$ is Zariski open subset of $X$, then for all non-trivial,
irreducible $R$-modules $W$, the restriction mapping
$$
H^j(X,\V) \to H^j(U,\V)
$$
is an isomorphism when $j=0,1$ and an injection when $j=2$.
\end{theorem}

\begin{proof}
Write
$$
X-U = Z \cup \bigcup_\alpha D_\alpha
$$
where $Z$ is a closed subvariety of $X$, each of whose components has
codimension $\ge 2$, and where each $D_\alpha$ is an irreducible divisor in $X$.
Since $W$ is a non-trivial irreducible representation of $R$,
Lemma~\ref{lem:families} implies that $H^0(D^o_\alpha,\W)$ vanishes for all
$\alpha$, where $D_\alpha^o$ denotes the smooth locus of $D_\alpha$. The result
now follows from the exactness of the Gysin sequence
$$
0 \to H^1(X,\W) \to H^1(U,\W) \to
\bigoplus_\alpha H^0(D^o_\alpha,\pi_\alpha^\ast\W)
\to H^2(X,\W) \to H^2(U,\W).
$$
\end{proof}

Since all irreducible representations of $\Sp(H)$ are defined over $\Q$, we
conclude:

\begin{corollary}
Suppose that $g=3$ or $g\ge 5$, $n\ge 0$ and $m\ge 1$. If $U$ is a Zariski open
subset of $\cC_g^n[m]$, then for all non-trivial, irreducible representations
$W$ of $\Sp(H_\Q)$, the map
$$
H^j(\cC_g^n[m],\W) \to H^j(U,\W)
$$
induced by the inclusion $U\hookrightarrow \cC_g^n[m]$
is an isomorphism when $j=0,1$ and an injection when $j=2$.
\end{corollary}

Since the monodromy representation $\pi_1(\cC_g^n[m],x_o)\to \Sp(H_\Q)$ is
Zariski dense, and since each irreducible representation of $\Sp(H_\Q)$ is
absolutely irreducible (so irreducible over $\Qp$ as well), the corollary holds
when $\Qp$ is replaced by any field of characteristic zero. In particular,
it holds when $\Qp$ is replaced by $\Q$.

\section{Proof of the Strictness Theorem}
\label{sec:proof}

To prove Theorem~\ref{thm:strictness}, it suffices to prove the corresponding
statement for $M^c$. Namely, that if $D$ is an irreducible divisor in $M^c$ and
$h:E \to D$ is a dominant morphism from an irreducible smooth variety to $D$,
then the image of $\pi_1(E,e_o)$ in $\G_o$ has finite index in the image of
$\pi_1(M^c,x_o)$ in $\G_o$. This suffices because $D\cap M$ is non-empty, as
each component of $M^c - M$ has codimension $\ge 2$, and because
$\pi_1(h^{-1}(D),e_o) \to \pi_1(E,e_o)$ is surjective as $E$ is smooth.

\subsection{Setup and Preliminaries}

The assumption on the Picard group of $M$ implies that $M=M^c=\P^1$ or $\dim M
\ge 2$. Since $\P^1$ is simply connected, the theorem is trivially true when
$\dim M = 1$. So we assume that $\dim M\ge 2$. Suppose that $D$ is an
irreducible divisor in $M^c$. Let $X$ be a generic $2$-dimensional linear
section of $\Mbar$ (with respect to any projective imbedding) that is transverse
to all strata of $(\Mbar-M)\cup D$. Since each component of $\Mbar-M^c$ has
codimension $\ge 3$, $X$ is a complete subvariety of $M^c$. Since each component
of $M^c-M$ has codimension $\ge 2$, all singularities of $X$ are isolated. The
local Lefschetz Theorem \cite[p.~153]{smt} implies that $X$ is unibranch at each
of its singular points.

Recall that $M_m \to M$ is the connected Galois covering that corresponds to the
kernel of the mod-$m$ monodromy homomorphism $\pi_1(M,x_o) \to \Aut(V_o/mV_o)$.
Define $X_m$ to be the covering of $X$ that corresponds to the kernel of
$\pi_1(X,x_o) \to \Aut(V_o/mV_o)$. It is a subvariety of $M_m$.

Define $\Mbar_m$ to be the normalization of $\Mbar$ in the function field of
$M_m$. The morphism $\Mbar_m \to \Mbar$ is finite and surjective. Define $M^c_m$
to be the inverse image of $M^c$ in $\Mbar_m$. The projection $M_m^c \to M^c$ is
a Galois covering. It follows that $M_m^c$ is also unibranch at each of its
points.

\begin{lemma}
\label{lem:X2M}
The homomorphism $\pi_1(X_m,x_o) \to \pi_1(M^c_m,x_o)$ is surjective for all
$m\ge 1$.
\end{lemma}

\begin{proof}
Set $\Xo_m = X_m\cap M_m$. Since $X_m$ and $M_m^c$ are unibranch at each of
their points, Lemma~\ref{lem:surjective} implies that the top and bottom
morphisms in the diagram
$$
\xymatrix{
\pi_1(\Xo_m,x_o) \ar@{->>}[r]\ar@{->>}[d] & \pi_1(X_m,x_o) \ar[d] \cr
\pi_1(M_m,x_o) \ar@{->>}[r] & \pi_1(M^c_m,x_o)
}
$$
are surjective. The Theorem~\ref{thm:smt} implies that the left-hand vertical
map is surjective. The result follows.
\end{proof}

The fact that every component of $M^c_m - M_m$ has codimension $\ge 2$ implies
that every irreducible divisor in $M^c$ is ample.

\begin{lemma}
\label{lem:ample}
For all $m\ge 1$, every irreducible divisor in $M_m^c$ is ample. Moreover, $\Pic
M_m^c$ is finitely generated and the restriction mapping $(\Pic M_m^c)\otimes \Q
\to (\Pic M_m)\otimes \Q$ is an isomorphism.
\end{lemma}

\begin{proof}
Fix an imbedding $M_m^c \hookrightarrow \P^r$. Suppose that $W$ is an
irreducible divisor in $M_m^c$. Since each component of $M_m^c-M_m$ has
codimension $\ge 2$, $W' := W\cap M_m$ is non-empty. Corollary~\ref{cor:ample}
implies that $W'$ is ample in $M_m$. Since $\Pic M_m$ has rank 1, there are
positive integers $d$ and $m$ such that $\cO_{M_m}(mW') \cong \cO_{M_m}(d)$. A
section of $\cO_{M_m}(mW')$ is a rational section of $\cO_{M_m^c}(d)$ over
$M_m^c$ that is regular on $M_m$. Suppose that $s\in H^0(M_m,\cO_{M_m}(d))$
satisfies $\div(s) = mW'$. Since all components of $M_m^c-M_m$ have codimension
$\ge 2$, $s$ must be regular on $M_m^c$ and have divisor $mW$. So $W$ is ample
in $M_m^c$. The second assertion is easily proved and is left to the reader.
\end{proof}

\begin{corollary}
\label{cor:F2M}
If $W$ is an irreducible divisor in $M_m^c$ such that $F = X_m \cap W$ is a
complete curve, then $\pi_1(F,x_o) \to \pi_1(X_m,x_o)$ and $\pi_1(F,x_o) \to
\pi_1(M_m^c,x_o)$ are surjective.
\end{corollary}

\begin{proof}
Lemma~\ref{lem:ample} implies that $W$ is an ample divisor in $M^c$. Let
$$
X_F := X_m - \{x\in X_m^\sing : x\notin F\}.
$$
Then $F\subset X_F$. Since $X_F - F$ is smooth, it is a local complete
intersection. The Lefschetz Theorem with Singularities \cite[p.~153]{smt}
implies that $\pi_1(F,x_o) \to \pi_1(X_F,x_o)$ is surjective. Since $X_m$ is
unibranch at each point of $X_m-X_F$, Lemma~\ref{lem:surjective} implies that
$\pi_1(X_F,x_o) \to \pi_1(X_m,x_o)$ is surjective. Consequently, $\pi_1(F,x_o)
\to \pi_1(X_m,x_o)$ is also surjective. The final assertion follows from
Lemma~\ref{lem:X2M}.
\end{proof}

\subsection{Proof of Theorem~\ref{thm:strictness}}

Set $C$ = $D\cap X$. Denote its normalization by $\pi:\Ctilde \to C$. We may
take $x_o$ to be a smooth point of $C$. Denote its inverse image in $\Ctilde$ by
$\xtilde_o$. Because $C$ may have multiple branches at the singular points of
$X$, the homomorphism $\pi_1(\Ctilde,\xtilde_o) \to \pi_1(C,x_o)$ may not be
surjective.

We will show that the image of $\pi_1(\Ctilde,\xtilde_o)^\wedge \to \G_o^\wedge$
has finite index in the image of the monodromy representation $\rho :
\pi_1(C,x_o)^\wedge \to \G_o^\wedge$, even though the image of
$\pi_1(\Ctilde,x_o)^\wedge$ in $\pi_1(C,x_o)^\wedge$ may have infinite index in
$\pi_1(C,x_o)^\wedge$. This is the only point in the proof where we work with
profinite fundamental groups. It is here where we use the assumption that the
Picard number is 1.

\begin{lemma}
\label{lem:components}
Suppose that $m\ge 1$. If $C'$ is a component of the inverse image $C_m$ of $C$
in $X_m$, then
\begin{enumerate}

\item $C'$ is an ample curve in $X_m$,

\item \label{it:degree}
$[C_m] = d [C'] \in (\Pic X_m)\otimes \Q$, where $d$ is the number of
irreducible components of $C_m$,

\item $\pi_1(C',x_o) \to \pi_1(M^c_m,x_o)$ is surjective.

\end{enumerate}
\end{lemma}

\begin{proof}
Denote the inverse image of $D$ in $M_m$ by $D_m$. This may be reducible. Let
$W$ the irreducible component of $D_m$ that contains $C'$. Since $X$ is a
generic linear section of $M_m^c$, Bertini's Theorem \cite[p.~151]{smt} implies
that $X_m\cap W$ is irreducible and therefore equal to $C'$.
Lemma~\ref{lem:ample} implies that $W$ is ample in $M_m$, which implies that
$C'$ is ample in $X_m$.

Since $\Pic M_m$ has rank $1$, the classes $[D_m]$ and $[W]$ are proportional in
$\Pic M_m$ mod torsion. Their restrictions $[C_m]$ and $[C']$ to $X_m$ are
therefore proportional mod torsion. The second assertion follows as $X_m \to X$
is a Galois covering and the Galois group acts transitively on the components of
$C_m$. The last assertion follows from Corollary~\ref{cor:F2M}.
\end{proof}

For $m\ge 1$, set $\G_{o,m} = \Aut (V_o/mV_o)$. This is a finite group
isomorphic to $\GL_N(\Z/m\Z)$.

The formulation of the following lemma and its proof were contributed by Madhav
Nori. They avoid a gap in the proof of a previous version of the lemma.

\begin{lemma}
\label{lem:normalization}
For all $m\ge 1$, the index of the image of $\pi_1(\Ctilde_o,\xtilde_o) \to
\G_{o,m}$ in the image of $\pi_1(C,x_o) \to \G_{o,m}$ is bounded by $C\cdot C$.
\end{lemma}

\begin{proof}
Set
$$
G_m = \im\{\pi_1(C,x_o) \to \G_{o,m}\}
\text{ and }
K_m = \im\{\pi_1(\Ctilde_o,\xtilde_o) \to \G_{o,m}\}. 
$$
Corollary~\ref{cor:F2M} implies that $G_m$ is also the image of $\pi_1(X_m,x_o)$
in $\G_{o,m}$. Consequently, the inverse image $C_m$ of $C$ in $X_m$ is a
connected Galois covering of $C$. Choose a lift $x_o'$ of $x_o$ to $C_m$.
Denote the component of $C_m$ that contains
$x_o'$ by $C'$.

The group $G_m$ acts on $X_m$ with quotient $X$. The stabilizer of the component
$C'$ is $K_m$. Let $Y_m = K_m\bs X_m$ so that we have the tower
$$
\xymatrix{
X_m \ar[r] & Y_m \ar[r]^h & X
}
$$
of unramified coverings.  Denote the image of $C'$ in $Y_m$ by $C''$. Since
$K_m$ is the stabilizer of $C'$, the morphism $C'' \to C$ induced by $h$ is
birational. Lemma~\ref{lem:components}(\ref{it:degree}) implies that
$$
h^\ast[C] = (\deg h) [C'] = [G_m:K_m][C'] \in (\Pic Y_m)\otimes\Q
$$
and that $C''$ is ample. Since $\deg h = [G_m:K_m]$, the projection formula
implies that
$$
[G_m:K_m] (C''\cdot C'') = (C''\cdot h^\ast C) = (C\cdot C).
$$
Since $C''$ is ample in $Y_m$, $C''\cdot C'' > 0$, which gives the desired
inequality
$$
[G_m:K_m] = \frac{C\cdot C}{C''\cdot C''} \le C\cdot C.
$$
\end{proof}

Combining this with Corollary~\ref{cor:F2M}, we obtain:

\begin{corollary}
\label{cor:normalization}
For all $m\ge 1$, the index of the image of $\pi_1(\Ctilde,\xtilde_o)^\wedge \to
\G_o^\wedge$ in the image of $\pi_1(M^c,x_o)^\wedge \to \G_o^\wedge$ is finite
and bounded by $C\cdot C$. \qed
\end{corollary}

Denote the normalization of $D$ by $\Dtilde \to D$. Since $X$ is a generic
linear section of $M^c$ and $D$, $C$ is not contained in the singular locus of
$D$. It follows that there is a smooth Zariski open subset $U$ of $C$ that lies
in the smooth locus of $D$. It can therefore be regarded as a curve in
$\Dtilde$. Corollary~\ref{cor:normalization} and the commutativity of the
diagram
$$
\xymatrix{
\pi_1(U,x_o) \ar[r]\ar@{->>}[d] & \pi_1(\Dtilde,x_o)\ar[d] \cr
\pi_1(\Ctilde,\xtilde_o) \ar[r] & \pi_1(M^c,x_o)
}
$$
implies that $\im\{\pi_1(\Dtilde,x_o)^\wedge \to \G_o^\wedge\}$ has finite index
in $\im\{\pi_1(M^c,x_o)^\wedge \to \G_o^\wedge\}$. Theorem~\ref{thm:strictness}
now follows from Proposition~\ref{prop:normalization}.

\section{Counterexamples}
\label{sec:examples}

The following examples suggest that the hypotheses of
Theorem~\ref{thm:strictness} cannot easily be relaxed.

\begin{example}
Suppose that $g\ge 3$ and that $\ell\ge 3$. Suppose that $P\in \A_g[\ell]$. Let
$M = M^c =\Bl_P \A_g[\ell]$, the blow-up of $\A_g[\ell]$ at $P$. Take $\V$ to be
the standard local system over $\A_g[\ell]$. Take $E=D$ to be the exceptional
divisor. Then the image of $\pi_1(E,e_o)$ in $\Sp_g(\Z)$ is trivial, so that the
conclusions of Theorem~\ref{thm:strictness} do not hold in this case. All
hypotheses hold except for the condition (\ref{item6}) on Picard groups.
\end{example}

\begin{example}
One can also formulate a version of Theorem~\ref{thm:strictness} for orbifolds.
The genus $2$ example presented in Section~\ref{sec:genus2} with $\ell=1$ shows
that the conclusion of Theorem~\ref{thm:strictness} fail in this case. Here
$M=M^c=\A_2$. Condition (\ref{item6}), that $\rank\Pic M_m=1$, holds when $m=1$,
but fails for at least some $m>1$. The codimension of $\Mbar-M^c$ is 2, so
condition (\ref{item4}) also fails to hold.
\end{example}

The final example shows that if $D$ is an irreducible smooth hyperplane section
of a non-compact algebraic manifold $M$ of dimension $>2$, then the inclusion
$D\hookrightarrow M$ may not induce an isomorphism on fundamental groups. In
such cases, Theorem~\ref{thm:smt} implies that $D$ is not generic.

\begin{example}
Suppose that $\ell\ge 3$ is an odd integer. Since $\Sp_3(\Z)[\ell]\to
\Sp_3(\Z/2\Z)$ is surjective, the hyperelliptic locus $\cH_3[\ell]$ in
$\M_3[\ell]$ is irreducible. Since the level $\ell$ subgroup of the mapping
class group is torsion free, it follows that the level $\ell$ subgroup of the
hyperelliptic mapping class group is also torsion free. This implies that
$\cH_3[\ell]$ is a smooth divisor in $\M_3[\ell]$. Since $\Pic \A_3[\ell]$ has
rank 1, the locus of hyperelliptic jacobians (the closure of the image of
$\cH_3[\ell]$ in $\A_3[\ell]$) is a hyperplane section of $\A_3[\ell]$. It
follows that $\cH_3[\ell]$ is the inverse image of a hyperplane section of
$\A_3[\ell]$. The image of $\pi_1(\cH_3[\ell],x_o) \to \pi_1(\M_3[\ell],x_o)$
and its profinite completion have infinite index as the first group is the level
$\ell$ subgroup of the hyperelliptic mapping class group, which has infinite
index in the genus 3 mapping class group. However, a result of A'Campo
\cite{acampo} implies that the image of the monodromy homomorphism
$\pi_1(\cH_3[\ell],x_o) \to \Sp_3(\Z)$ has finite index in $\Sp_3(\Z)$. The
divisor $\cH_3[\ell]$ is not generic as, for example, its image $\cH_3$ in
$\M_3$ is not transverse to the the singular locus of the boundary divisor
$\Delta_0$ of $\overline{\M}_3$.
\end{example}

\section{Discussion of Strictness Theorems}
\label{sec:discussion}

Consider the diagram
$$
\xymatrix{
Z \ar[r]^f & Y \ar[r]^h & X
}
$$
of morphisms of varieties where $X$ and $Z$ are smooth and $f$ is dominant.
Techniques developed by Deligne in \cite[\S 8]{deligne:hodge} to prove
``strictness theorems'' imply the following result:

\begin{proposition}
\label{prop:abelian}
The image of $H_1(Z;\Z) \to H_1(X;\Z)$ has finite index in the image of
$H_1(Y;\Z) \to H_1(X;\Z)$.
\end{proposition}

Since this result does not appear explicitly in the literature, we sketch a
proof in the next section.

It would be very useful to know the extent to which this strictness result
extends to (topological, profinite, proalgebraic, \dots) fundamental groups.
Choose base points $x_o$, $y_o$ and $z_o$ such that $f$ and $h$ are basepoint
preserving. The most optimistic statement that one might hope to be true is that
the image of $\pi_1(Z,z_o) \to \pi_1(X,x_o)$ has finite index in the image of
$\pi_1(Y,y_o) \to \pi_1(X,x_o)$. If this were true, then
Theorem~\ref{thm:strictness} could be strengthened and its proof simplified.
Unfortunately, it is not true, as we will show by example below. However, weaker
statements that would follow from this optimistic statement, such as
Theorem~\ref{thm:strictness}, do hold. We shall call them {\em non-abelian
strictness theorems}. The only non-abelian strictness theorems of which I am
aware are due to Nori \cite[Thm.~WLT]{nori}, Lasell and Ramachandran
\cite[Thm.~1.1]{lasell-ramachandran}, and Napier and Ramachandran
\cite[Cor.~0.2]{napier-ramachandran}:

\begin{theorem}[Nori (1983), Napier-Ramachandran (1998)]
Suppose that $X$ and $Y$ are connected, smooth projective varieties of positive
dimension. If $Y \to X$ is a holomorphic immersion with ample normal bundle,
then the image of $\pi_1(Y,y_o)$ in $\pi_1(X,x_o)$ has finite index.
\end{theorem}

\begin{theorem}[Lasell-Ramachandran (1996)]
If $X$, $Y$ and $Z$ are all irreducible and proper, then for each positive
integer $N$ there is a finite quotient $\Delta_N$ of $\pi_1(Y,y_o)$ such that
for all fields $F$ and all reductive representations $\rho :\pi_1(X,x_o)\to
\GL_N(F)$ such that
$$
\xymatrix{
\pi_1(Z,z_o) \ar[r]^{f_\ast} & \pi_1(Y,y_o) \ar[r]^{h_\ast} &
\pi_1(X,x_o) \ar[r]^\rho & \GL_N(F)
}
$$
is trivial, the homomorphism $\rho\circ h_\ast : \pi_1(Y,y_o) \to \GL_N(F)$
factors through $\Delta_N$:
$$
\xymatrix{
\pi_1(Z,z_o)\ar[d]_{f_\ast} \ar[drr]^{\text{trivial}} \cr
\pi_1(Y,y_o) \ar@{->>}[d] \ar[r]^{h_\ast} & \pi_1(X,x_o) \ar[r]^\rho & \GL_N(F)
\cr
\Delta_N \ar@{..>}[urr]
}
$$
\end{theorem}

The following example shows that several of the most optimistic non-abelian
strictness statements given above are false.

\subsection{The abelian strictness theorem does not extend naively to $\pi_1$}
\label{sec:ab_strictness}

We give a general method of constructing counterexamples to the most general
forms of the non-abelian strictness assertion, and then use it to give an
explicit counter example. We will construct varieties
$$
\xymatrix{
Z \ar[r]^f & Y \ar[r]^h & X
}
$$
where $X$ and $Z$ are smooth and $f$ is dominant, and a homomorphism $\rho :
\pi_1(X,x_o) \to G(\Q)$ with the property that the dimension of the Zariski
closure of the image of $\pi_1(Z,z_o)$ in $G$ has dimension strictly smaller
than the dimension of the Zariski closure of the image of $\pi_1(Y,y_o)$ in $G$.
This provides counter examples to the discrete, profinite and algebraic versions
of the strictness assertion.

Suppose that $G$ is a semi-simple $\Q$-group of adjoint type such that the
symmetric space $D$ of $G(\R)$ is hermitian. Then $D$ has a complex structure
with the property that $G$ acts on $D$ by biholomorphisms. Suppose that $\G$ is
an arithmetic subgroup of $G(\Q)$. Denote by $D^o$ the open subset of $D$ on
which $\G$ acts fixed point freely. Suppose that $D^o$ is non-empty. Then $\G\bs
D$ is a quasi-projective variety, and $X:= \G\bs D^o$ is a smooth subvariety of
it.

Suppose that $H$ is a semi-simple $\Q$-subgroup of $G$, $H\neq G$, whose
associated symmetric space is hermitian. Suppose that $d_o \in D$ is a point
whose $H(\R)$-orbit $D_H$ satisfies:
\begin{enumerate}

\item the image of $D_H$ in $\G\bs X$ is a subvariety,

\item $d_o \in D_H^o := D_H\cap D^o$.

\end{enumerate}
Set $Z=(\G\cap H(\Q))\bs D_H^o$. Then $Z$ is non-empty and smooth, and the map
$Z \to X$ induced by $D_H\to D$ is finite. Its image $Y$ in $X$ is closed.
Denote the images of $d_o$ in $Z$, $Y$ and $X$ by $z_o$, $y_o$ and $x_o$, 
respectively. (Cf.\ Figure~\ref{fig:setup}.)
\begin{figure}[!ht]
\epsfig{file=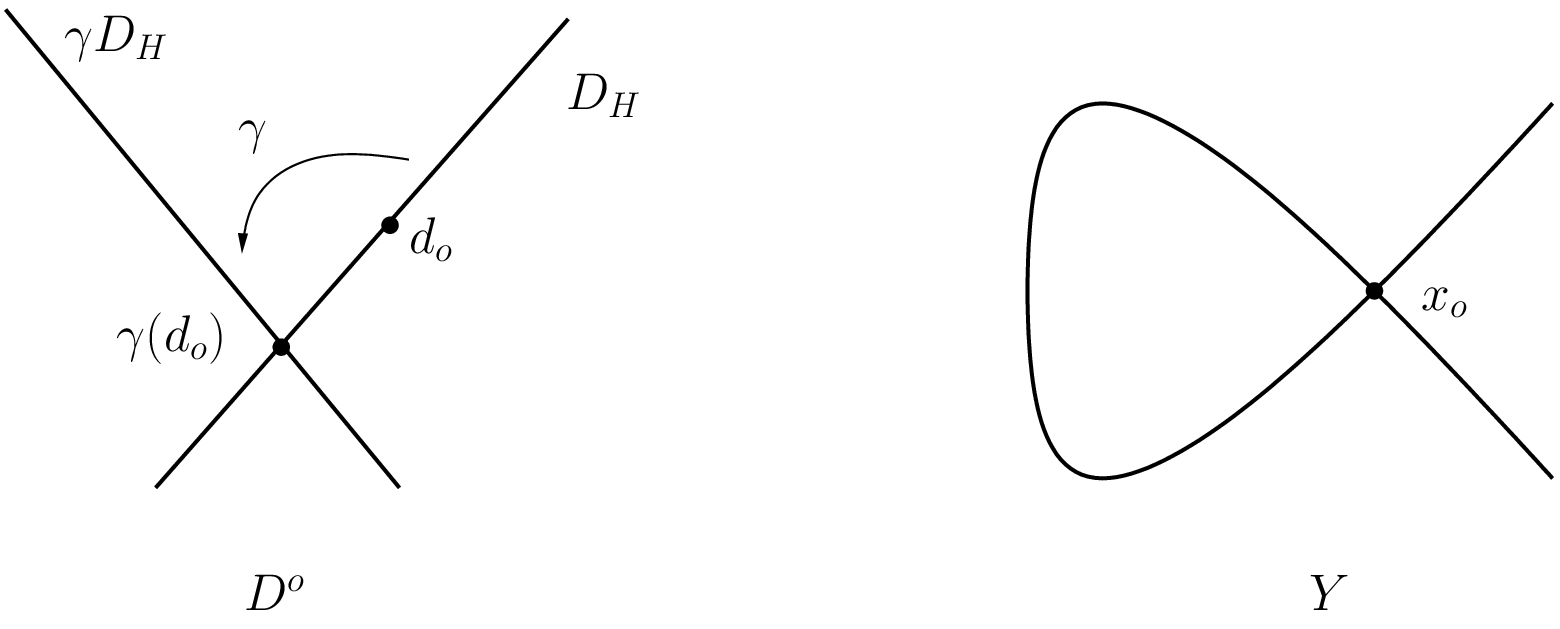, width=3.5in}
\caption{$Y$ and part of its inverse image in $D^o$}
\label{fig:setup}
\end{figure}

Note that there is an exact sequence
$$
1 \to \pi_1(D^o,d_o) \to \pi_1(X,x_o) \to \G \to 1.
$$
There is therefore a natural homomorphism $\pi_1(X,x_o) \to G(\Q)$.

\begin{proposition}
\label{prop:closure}
If there exists $\gamma \in \G$ such that 
\begin{enumerate}

\item $\gamma(d_o) \neq d_o$,

\item $\gamma(d_o) \in D_H$,

\item no positive power $\gamma^N$ of $\gamma$ lies in $H(\Q)$,

\end{enumerate}
then the dimension of the Zariski closure of the image of $\pi_1(Z,z_o)$ in $G$
is larger than the dimension of the Zariski closure of the image of
$\pi_1(Y,y_o)$ in $G$. Consequently, the image of $\pi_1(Z,z_o) \to
\pi_1(X,x_o)$ has infinite index in the image of  $\pi_1(Y,y_o) \to
\pi_1(X,x_o)$.
\end{proposition}

\begin{proof}
The image of $\pi_1(Z,z_o)$ in $G(\Q)$ is the arithmetic group $\G_H:=\G\cap
H(\Q)$. Since $\G_H$ is an arithmetic subgroup of the semi-simple group $H$, its
Zariski closure in $G$ is $H$. Since $\gamma^N \notin H(\Q)$ for all $N>0$,
$\gamma(D_H) \neq D_H$. Choose a path $\mu$ in $D_H^o$ from $d_o$ to
$\gamma(d_o)$. The image of $\mu$ in $Y$ is a loop whose image under
$\pi_1(Y,y_o) \to G(\Q)$ is $\gamma$. Since no positive power of $\gamma$ lies
in $H(\Q)$, the cosets $\{\gamma^N H(\Q):N\ge 1\}$ of $H$ in $G$ are distinct.
It follows that the Zariski closure of the image of $\pi_1(Y,y_o) \to G(\Q)$ has
dimension strictly larger than the dimension of $H(\Q)$.
\end{proof}

\begin{example}
Let $G$ be the $\Q$-group $\PSL_2\times\PSL_2$, $D$ be the product $\h\times\h$
of two copies of the upper half plane, and let $\G =
\PSL_2(\Z)\times\PSL_2(\Z)$. Let
$$
g_1 = \begin{pmatrix} 0 & {-1/2} \cr 1 & 1\end{pmatrix} \text{ and }
g_2 = \begin{pmatrix} 1 & 0 \cr 0 & 3 \end{pmatrix}.
$$
These are elements of $\GL_2(\Q)$ with positive determinant and therefore act on
$\h$. Set $d_o = (g_1(i),g_2(i))$. Since $-1/g_1(i) = 2 + 2i$ and $-1/g_2(i) =
3i$, the $\SL_2(\Z)$ orbits of $g_1(i)$ and $g_2(i)$ both intersect $\Im\tau >
1$, which implies that their isotropy groups in $\PSL_2(\Z)$ are trivial. Thus
$d_o \in D^o$.

Set $g=(g_1,g_2)$. Denote the diagonal copy of $\PSL_2$ in $\PSL_2\times\PSL_2$
by $\Delta_G$. Let $H = g \Delta_G g^{-1}$ and $D_H$ be the $H$-orbit of $d_o$.
Define $\gamma = (\gamma_1,\gamma_2)$, where
$$
\gamma_1 = \begin{pmatrix} 1 & 1 \cr 0 & 1\end{pmatrix} \text{ and }
\gamma_2 = \begin{pmatrix} 2 & -1 \cr -3 & 2 \end{pmatrix}.
$$
Since $\trace(\gamma_1) = 2$ and $\trace(\gamma_2) = 4$, both act fixed point
freely on $\h$. In particular, $\gamma(d_o) \neq d_o$. If $\gamma^N$ is
conjugate to an element of $H$, then there exists $h\in \PSL_2(\Q)$ such that
$\gamma_1^N = g_1 h^N g_1^{-1}$ and $\gamma_2^N = g_2 h^N g_2^{-1}$. In
particular, $\gamma_1^N$ is conjugate to $\gamma_2^N$ in $\PSL_2(\Q)$. Since
$\gamma_1$ has eigenvalue $1$, and $\gamma_2$ has eigenvalues $2\pm \sqrt{3}$,
$\gamma_1^N$ is conjugate to $\gamma_2^N$ if and only if $N=0$.
\end{example}

\subsection{Proof of the Abelian Strictness Assertion}

All cohomology in this section will be with rational coefficients. Recall
from \cite{deligne:hodge} that
the rational cohomology of a (complex algebraic) variety $T$ has a natural
weight filtration
$$
0 = W_{-1}H^j(T) \subseteq \cdots \subseteq W_r H^j(T) \subseteq
W_{r+1} H^j(T) \subseteq \cdots \subseteq W_{2j}H^j(T) = H^j(T)
$$
Its $r$th graded quotient $W_rH^j(T)/W_{r-1}$ will be denoted by $\Gr^W_r
H^j(T)$. It has the property that $W_{j-1}H^j(T)=0$ when $T$ is smooth, and
$H^j(T)=W_jH^j(T)$ when $T$ is proper.

If $f : S \to T$ is a morphism of complex algebraic varieties, then the induced
morphism
$$
f^\ast : H^\dot(T) \to H^\dot(S)
$$
is {\em strict} with respect to the weight filtration $W_\dot$. That is,
$$
\im\{f^\ast : H^j(T) \to H^j(S) \} \cap W_rH^j(T)
= \im\{f^\ast : W_r H^j(T) \to H^j(S)\}.
$$

\begin{proposition}
\label{prop:technical}
If $f : Z \to Y$ is a dominant morphism from a smooth variety, then
$$
\xymatrix{
0 \ar[r] & W_0 H^1(Y) \ar[r] & H^1(Y) \ar[r]^{f^\ast} & H^1(Z)
}
$$
is exact.
\end{proposition}

Proposition~\ref{prop:abelian} is an immediate consequence. Since $H_1(X;\Z)$,
$H_1(Y;\Z)$ and $H_1(Z;\Z)$ are finitely generated abelian groups, it suffices
to prove that the images of $H_1(Z;\Q)$ and $H_1(Y;\Q)$ in $H_1(X;\Q)$ are
equal. We will prove the dual assertion; namely:
$$
\ker\{h^\ast : H^1(X) \to H^1(Y)\} = \ker\{f^\ast h^\ast : H^1(X) \to H^1(Z)\}.
$$
To prove this, consider the diagram
$$
\xymatrix{
0 \ar[r] & W_0 H^1(Y) \ar[r] & H^1(Y) \ar[r]^{f^\ast} & H^1(Z) \cr
&& H^1(X) \ar[u]^{h^\ast} \ar[ur]
}
$$
Since $X$ is smooth, $W_0 H^1(X)=0$. Strictness and the exactness of the
top row of the diagram imply that
$$
\im h^\ast \cap \ker f^\ast = \im h^\ast \cap W_0 H^1(Y) = \im W_0 H^1(Z) = 0.
$$
This implies that $\im f^\ast = \im(f^\ast h^\ast)$ as required.

\begin{proof}[Sketch of Proof of Proposition~\ref{prop:technical}]
This proof requires familiarity with Deligne's construction \cite{deligne:hodge}
of the mixed Hodge structure (MHS) on the cohomology of a general complex
algebraic variety. The first step is to observe that we may assume that $Z\to Y$
is proper and surjective. If not, one can find a smooth variety $Z'$ that
contains $Z$ as a Zariski open subset and a proper surjective morphism $h' :
Z'\to Y$ that extends $h$. Since $Z'$ is smooth, standard arguments imply that
$H^1(Z')\to H^1(Z)$ is injective.

Suppose that $h : Z \to Y$ is surjective. Fix a completion $\Ybar$ of $Y$. Then
there exists a smooth completion $\Zbar$ of $Z$ such that $\Zbar-Z$ is a normal
crossings divisor $D$ in $\Zbar$, and a morphism $\hbar : \Zbar \to \Ybar$ that
extends $h$.

Set $Y_0 = \Ybar - Y$. Then $(\Zbar,D) \to (\Ybar,Y_0)$ can be completed to a
hypercovering $(\Mbar_\dot,D_\dot)$ of $(\Ybar,Y_0)$ where
\begin{enumerate}

\item each $\Mbar_n$ is
smooth and proper,

\item  $D_n$ is a normal crossings divisor in $\Mbar_n$,

\item $(M_0,D_0) = (\Zbar,D)$.

\end{enumerate}
Set $M_n = \Mbar_n - D_n$.
The spectral sequence of the associated simplicial variety satisfies
$$
E_1^{s,t} = H^t(M_s) \implies H^{s+t}(Y).
$$
Since it is a spectral sequence in the category of MHS
$$
H^1(Y)/W_0 = \ker\{H^1(M_0) \overset{d_1}{\longrightarrow} H^1(M_1)\}.
$$
The result follows as $Z=M_0$.
\end{proof}

\end{document}